\documentclass[12pt]{article}
\usepackage{a4wide,epsfig,amsmath,amssymb,latexsym,color}
\usepackage{subfig}

\usepackage{epstopdf}
\def\RR{\mathbb{R}}
\def\SS{\mathbb{S}}

\def\ehat{\widehat{e}}

\def\bfn{\mathbf{n}}

\def\bfv{\mathbf{v}}

\def\bfx{\mathbf{x}}
\def\bfy{\mathbf{y}}

\def\sgn{{\rm sgn}}

\def\bfmu{{\boldsymbol \mu}}

\def\calI{\mathcal{I}}

\newtheorem{theorem}{Theorem}

\newenvironment{proof}{\begin{trivlist}\item[]{\emph{Proof.}}}
               {\hfill$\Box$\end{trivlist}}

\def\Acknowledgement{\goodbreak\bigskip\noindent{\bf Acknowledgement.\ }}

\begin{document}

\title{Transfinite mean value interpolation \\
       over polygons}
\author{Michael S. Floater\footnote{
   Department of Mathematics,
   University of Oslo, Moltke Moes vei 35, 0851 Oslo, Norway,
   {\it email: michaelf@math.uio.no}}
\and
Francesco Patrizi\footnote{
   SINTEF, PO Box 124 Blindern, 0314 Oslo, Norway,
   {\it email: francesco.patrizi@sintef.no}}
}

\maketitle

\abstract{
Mean value interpolation is a method for fitting a smooth function
to piecewise-linear data prescribed on the boundary of a polygon
of arbitrary shape, and has applications in
computer graphics and curve and surface modelling.
The method generalizes to transfinite interpolation, i.e.,
to any continuous data on the boundary but a mathematical proof
that interpolation always holds has so far been missing.
The purpose of this note is to complete this gap in the theory.
}

\smallskip

\noindent {\em Keywords: }
Mean value coordinates, mean value interpolation,
transfinite interpolation.

\section{Introduction}
One of the main uses of generalized barycentric coordinates (GBCs)
is to interpolate piecewise-linear data prescribed on the
boundary of a polygon with a smooth function.
This kind of barycentric interpolation has been used, for example,
in computer graphics, as the basis for
image warping, and in higher dimension, for mesh deformation.

One type of GBC that is frequently used for this is
mean value (MV) coordinates due to a simple closed formula.
MV coordinates have been studied extensively in
various papers~\cite{Floater:15a}
but while they are simple to implement,
a mathematical proof of interpolation seems surprisingly difficult.
A proof for convex polygons is relatively
simple and follows from the fact that MV coordinates
are positive in this case.
Interpolation for a convex polygon holds in fact for
any positive barycentric coordinates; see \cite{Floater:06}.
For arbitrary polygons, a specific proof of interpolation
for MV coordinates was derived in \cite{Hormann:06}.

The MV interpolant to piecewise-linear boundary data is based on
integration with respect to angles around each chosen point
inside the polygon.
This construction extends in a natural way to
any continuous boundary data thus providing
a transfinite interpolant~\cite{Ju:05a,Dyken:09}.
Such interpolation could have various applications, one of which is
its use as a building block for
interpolants of higher order that also match derivative
data on the boundary.
However, there is currently no mathematical proof of
interpolation in the transfinite setting in all cases,
only numerical evidence. Like in the piecewise-linear case,
when the polygon is convex, interpolation is easier to establish.
In fact it was shown in \cite{Dyken:09} for more general domains,
convex or otherwise, under the condition that the distance between the
external medial axis of the domain and the domain boundary
is strictly positive.
This latter condition trivially holds for convex domains
since there is no external medial axis in this case.

This still leaves open the question of whether MV interpolation
really interpolates any continuous data on the boundary of
an arbitrary polygon, and this is what we establish in this paper.
The proof parallels that of \cite{Hormann:06}
in that we treat interpolation at edge points and vertices
separately: in Theorems~\ref{thm:edge} and~\ref{thm:vertex}
respectively.
At the end of the paper we give two examples
that numerically confirm the interpolation property.

In the future we would like to extend the proof of interpolation
to 3D geometry such as volumes enclosed by
triangular meshes~\cite{Floater:05,Ju:05a}
but there does not seem to be any straightforward generalization
of the proof in the 2D case, not even for piecewise-linear boundary data.
It would also be interesting to establish transfinite interpolation
over more general domains with weaker conditions on the
shape of the boundary than those used in~\cite{Dyken:09}.

\section{Definitions}
Let $\Omega \subset \RR^2$ be a polygon with vertices $V$ and edges $E$.
Suppose that $f: \partial \Omega \to \RR$ is a continuous function
on the boundary $\partial \Omega$.
We define a function $g:\Omega \to \RR$
as follows.
For each edge $e \in E$, let $\bfn_e$ denote the
outward unit normal to $e$ with respect to $\Omega$,
and for each point
$\bfx \in \Omega$, let
$h_e(\bfx)$ be its signed distance to $e$,
$$ h_e(\bfx) = (\bfy - \bfx) \cdot \bfn_e, $$
for any $\bfy \in e$.
We let $\tau_e(\bfx) \in \{-1,0,1\}$ be the sign of the distance,
$$ \tau_e(\bfx) = \sgn(h_e(\bfx)). $$

Let $\SS_1$ denote the unit circle in $\RR^2$.
For $\bfx \in \Omega$, let
$\ehat(\bfx) \subset \SS_1$ denote the circular arc on $\SS_1$
formed by projecting $e$ onto the unit circle centred at $\bfx$,
$$ \ehat(\bfx) = \left\{\frac{\bfy - \bfx}{\|\bfy - \bfx\|}
                        : \bfy \in e \right\}, $$
with $\| \cdot \|$ the Euclidean norm.
This arc is just a point in the case that $\tau_e(\bfx)=0$.
Suppose $\tau_e(\bfx) \ne 0$.
Then for each unit vector $\bfmu \in \ehat(\bfx)$, let
$\bfy_e(\bfx,\bfmu)$ be the unique point of $e$ such that
$$ \frac{\bfy_e(\bfx,\bfmu) - \bfx}
        {\|\bfy_e(\bfx,\bfmu) - \bfx\|} = \bfmu, $$
and let
$$ I_e(\bfx) = 
   \int_{\ehat(\bfx)}
       \frac{1}
            {\|\bfy_e(\bfx,\bfmu) - \bfx\|} \, d \bfmu > 0, \qquad
   I_e(\bfx;f) = 
   \int_{\ehat(\bfx)}
       \frac{f(\bfy_e(\bfx,\bfmu))}
            {\|\bfy_e(\bfx,\bfmu) - \bfx\|} \, d \bfmu. $$
In the case that $\tau_e(\bfx)=0$,
we define $I_e(\bfx) =  I_e(\bfx;f) =  0$.

We now define
\begin{equation}\label{eq:mv}
 g(\bfx) = \calI f(\bfx) =
   \sum_{e \in E} \tau_e(\bfx) I_e(\bfx;f) \Big/ \phi(\bfx), 
\end{equation}
where
\begin{equation}\label{eq:phi}
 \phi(\bfx) = \sum_{e \in E} \tau_e(\bfx) I_e(\bfx).
\end{equation}
As shown in \cite{Floater:03a},
if $e = [\bfv_1,\bfv_2]$ then
\begin{equation}\label{eq:I1}
 I_e(\bfx) = \tan(\alpha_e(\bfx)/2)
                 \left(\frac{1}{\|\bfv_1-\bfx\|}
                   + \frac{1}{\|\bfv_2-\bfx\|} \right),
\end{equation}
where
$\alpha_e(\bfx) \in [0,\pi)$ is the angle at
$\bfx$ of the triangle $[\bfx,\bfv_1,\bfv_2]$.
It was shown in \cite{Hormann:06} that $\phi(\bfx) > 0$ for all
$\bfx \in \Omega$, and in the case that $f$ is linear,
$g$ interpolates $f$.

\section{Interpolation on an edge}\label{sec:edge}

\begin{theorem}\label{thm:edge}
Let $\bfy_*$ be an interior point
of some edge of $\partial \Omega$.
Then
$g(\bfx) \to f(\bfy_*)$ as $\bfx \to \bfy_*$ for
$\bfx \in \Omega$.
\end{theorem}

\begin{proof}
From the form of (\ref{eq:mv}),
$$ g(\bfx) - f(\bfy_*) =
   \sum_{e \in E} \tau_e(\bfx) I_e(\bfx;\tilde f) \Big/ \phi(\bfx), $$
where $\tilde f(\bfy) := f(\bfy) - f(\bfy_*)$ and therefore
\begin{equation}\label{eq:gfbound}
 |g(\bfx) - f(\bfy_*)| \le
   \sum_{e \in E} I_e(\bfx;|\tilde f|) \Big/ \phi(\bfx).
\end{equation}

Let $[\bfv_1,\bfv_2] \in E$ be the edge containing $\bfy_*$,
as in Figure~\ref{fig:edge}.
\begin{figure}[ht]
\centering
\includegraphics[width=0.40\textwidth]{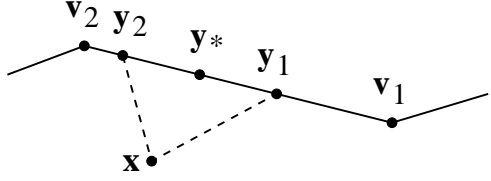}
\caption{Interpolation at an edge point $\bfy_*$.}
\label{fig:edge}
\end{figure}
Let $\epsilon > 0$.
By the continuity of $f$, there is some $\delta$, where
$$ 0 < \delta < \min\{\|\bfv_1-\bfy_*\|,\|\bfv_2-\bfy_*\|\}, $$
such that if $\bfy \in [\bfv_1,\bfv_2]$ and $\|\bfy - \bfy_*\| \le \delta$
then $|f(\bfy)-f(\bfy_*)| < \epsilon$.
Let $\bfy_j \in [\bfy_*,\bfv_j]$, $j=1,2$, be the point such that
$\|\bfy_j-\bfy_*\| = \delta$, and let
$e_0 = [\bfy_1,\bfy_2]$.
Then,
$$ \sum_{e \in E} I_e(\bfx;|\tilde f|)
 = I_{e_0}(\bfx;|\tilde f|)
  + \sum_{e \in F} I_e(\bfx;|\tilde f|),
$$
where
$$ F = \{[\bfv_1,\bfy_1],[\bfy_2,\bfv_2]\} 
  \cup (E \setminus [\bfv_1,\bfv_2]), $$
and it follows that
$|g(\bfx) - f(\bfy_*)| \le \gamma(\bfx) / \phi(\bfx)$, where
$$ \gamma(\bfx) = \epsilon I_{e_0}(\bfx)
   + 2 M \sum_{e \in F} I_e(\bfx), $$
and
\begin{equation}\label{eq:M}
 M := \sup_{\bfy \in \partial \Omega} |f(\bfy)|.
\end{equation}
Similar to $\gamma(\bfx)$, we can express $\phi(\bfx)$ as
$$ \phi(\bfx) = \tau_{e_0}(\bfx) I_{e_0}(\bfx)
   + \sum_{e \in F} \tau_e(\bfx) I_e(\bfx). $$
For $\bfx$ close enough to $\bfy_*$,
$\tau_{e_0}(\bfx) = 1$, and then
$$ \frac{\gamma(\bfx)}{\phi(\bfx)} =  \frac{\epsilon + 
    2 M \sum_{e \in F} I_e(\bfx) / I_{e_0}(\bfx)}
    {1 + \sum_{e \in F} \tau_e(\bfx) I_e(\bfx) 
    / I_{e_0}(\bfx)}. $$
As $\bfx \to \bfy_*$, $\alpha_{e_0}(\bfx) \to \pi$, and
since $\bfy_* \not \in e$ for all $e \in F$,
$$ \alpha_{e}(\bfx) \to \alpha_{e}(\bfy_*) < \pi, \quad e \in F. $$
Therefore, by (\ref{eq:I1}), as $\bfx \to \bfy_*$,
$$ I_{e_0}(\bfx) \to \infty \quad\hbox{and}\quad
 I_{e}(\bfx) \to I_{e}(\bfy_*) \ne \infty, \quad e \in F. $$
Thus $\gamma(\bfx)/\phi(\bfx) \to \epsilon$
as $\bfx \to \bfy_*$.
Hence,
$$ \limsup_{\bfx \to \bfy_*} |g(\bfx) - f(\bfy_*)| \le \epsilon $$
for any $\epsilon > 0$
which shows that $|g(\bfx) - f(\bfy_*)| \to 0$
as $\bfx \to \bfy_*$.
\end{proof}

\section{Interpolation at a vertex}\label{sec:vertex}

\begin{theorem}\label{thm:vertex}
For $\bfv \in V$, $g(\bfx) \to f(\bfv)$ as $\bfx \to \bfv$
for $\bfx \in \Omega$.
\end{theorem}

\begin{proof}
Similar to (\ref{eq:gfbound}), from the form of (\ref{eq:mv}),
$$ |g(\bfx) - f(\bfv)| \le
   \sum_{e \in E} I_e(\bfx;|\tilde f|) \Big/ \phi(\bfx), $$
where $\tilde f(\bfy) := f(\bfy) - f(\bfv)$.

Let $\bfv_1$ and $\bfv_2$ be the two neighbouring vertices of
$\bfv$ with $\bfv_1,\bfv,\bfv_2$ ordered
anticlockwise w.r.t.\ $\partial \Omega$
as in Figures~\ref{fig:convex} and~\ref{fig:concave}.
\begin{figure}[ht]
\centering
\includegraphics[width=0.32\textwidth]{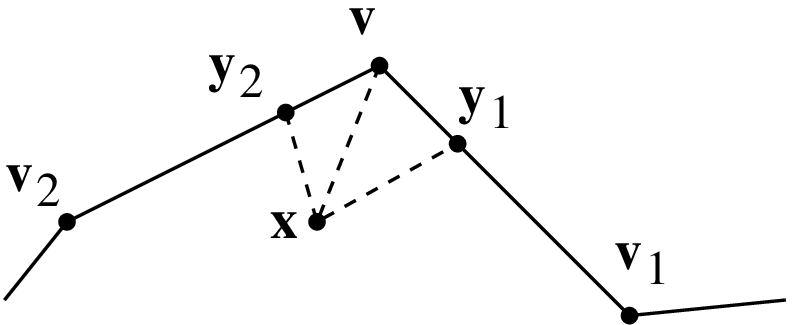}
\caption{Interpolation at a convex vertex $\bfv$.}
\label{fig:convex}
\end{figure}
\begin{figure}[ht]
\centering
\includegraphics[width=0.28\textwidth]{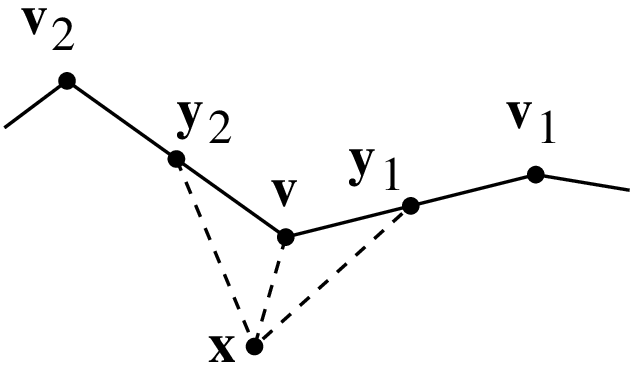}
\quad
\includegraphics[width=0.28\textwidth]{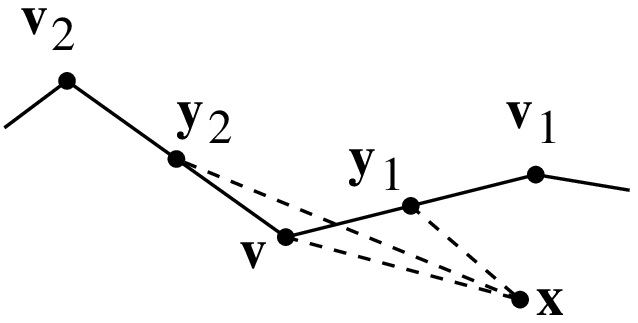}
\quad
\includegraphics[width=0.28\textwidth]{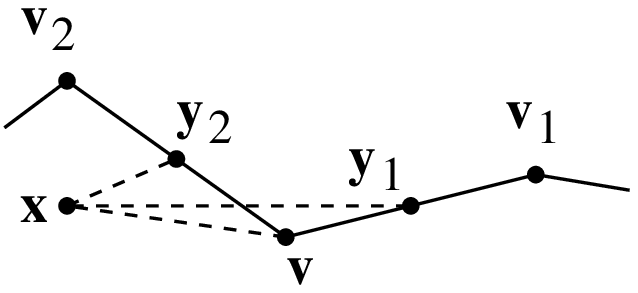}
\caption{Interpolation at a concave vertex $\bfv$.}
\label{fig:concave}
\end{figure}
Let $\epsilon > 0$.
By the continuity of $f$, there is some $\delta$, where
$$ 0 < \delta < \min\{\|\bfv_1-\bfv\|,\|\bfv_2-\bfv\|\}, $$
such that if $\bfy$ is in $[\bfv_1,\bfv]$ or $[\bfv,\bfv_2]$
and $\|\bfy - \bfv\| \le \delta$
then $|f(\bfy)-f(\bfv)| < \epsilon$.
Let $\bfy_j \in [\bfv,\bfv_j]$, $j=1,2$, be the point such that
$\|\bfy_j-\bfv\| = \delta$, and define
$e_1 = [\bfy_1,\bfv]$ and $e_2 = [\bfv,\bfy_2]$.
Then,
$$ \sum_{e \in E} I_e(\bfx;|\tilde f|)
  = I_{e_1}(\bfx;|\tilde f|) + I_{e_2}(\bfx;|\tilde f|)
  + \sum_{e \in F} I_e(\bfx;|\tilde f|),
$$
where
$$ F = \{[\bfv_1,\bfy_1],[\bfy_2,\bfv_2]\}
       \cup (E \setminus \{[\bfv_1,\bfv],[\bfv,\bfv_2]\}). $$
It follows that
$|g(\bfx) - f(\bfy_*)| \le \gamma(\bfx) /\phi(\bfx)$, where
$$ \gamma(\bfx) = 
     \epsilon (I_{e_1}(\bfx) + I_{e_2}(\bfx))
   + 2 M \sum_{e \in F} I_e(\bfx), $$
and $M$ is as in (\ref{eq:M}).
We can similarly express $\phi(\bfx)$ as
$$ \phi(\bfx) =
    \tau_{e_1}(\bfx) I_{e_1}(\bfx) + \tau_{e_2}(\bfx) I_{e_2}(\bfx)
    + \sum_{e \in F} \tau_e(\bfx) I_e(\bfx). $$
Then using (\ref{eq:I1}),
and multiplying both $\gamma(\bfx)$ and $\phi(\bfx)$ by
$\|\bfv - \bfx\|$, we have
$$
 \frac{\gamma(\bfx)}{\phi(\bfx)} = 
 \frac{\epsilon(
   \tan(\alpha_{e_1}(\bfx)/2) + \tan(\alpha_{e_2}(\bfx)/2)) + A(\bfx)}
      {\tau_{e_1}(\bfx) \tan(\alpha_{e_1}(\bfx)/2)
   + \tau_{e_2}(\bfx) \tan(\alpha_{e_2}(\bfx)/2) + B(\bfx)}, $$
where $A(\bfx), B(\bfx) \to 0$ as $\bfx \to \bfv$.
Letting $\tau_j = \tau_{e_j}$ and
$\alpha_j = \alpha_{e_j}$, $j=1,2$, and
using the fact that $-\tan(\beta) = \tan(-\beta)$ for $\beta \in \RR$,
we can rewrite this as
$$
 \frac{\gamma(\bfx)}{\phi(\bfx)} = 
 \frac{\epsilon(
   \tan(\alpha_1(\bfx)/2) + \tan(\alpha_2(\bfx)/2)) + A(\bfx)}
      {\tan(\tau_1(\bfx) \alpha_1(\bfx)/2)
   + \tan(\tau_2(\bfx)\alpha_2(\bfx)/2) + B(\bfx)}. $$
Next, using the identity
$$ \tan(\beta_1) + \tan(\beta_2) 
   = \frac{\sin(\beta_1+\beta_2)} {\cos(\beta_1)\cos(\beta_2)}, $$
and the fact that $\cos(-\beta) = \cos(\beta)$, it follows that
$$ \frac{\gamma(\bfx)}{\phi(\bfx)} = 
 \frac{\epsilon \sin((\alpha_1(\bfx) 
   + \alpha_2(\bfx))/2) + \tilde A(\bfx)}
 {\sin((\tau_1(\bfx)\alpha_1(\bfx) 
  + \tau_2(\bfx)\alpha_2(\bfx))/2) + \tilde B(\bfx)}, $$
where
\begin{align*}
 \tilde A(\bfx) & = \cos((\alpha_1(\bfx)/2) 
                    \cos((\alpha_2(\bfx)/2) A(\bfx), \cr
 \tilde B(\bfx) & = \cos((\alpha_1(\bfx)/2) 
                    \cos((\alpha_2(\bfx)/2) B(\bfx),
\end{align*}
and so also $\tilde A(\bfx), \tilde B(\bfx) \to 0$ as $\bfx \to \bfv$.

Finally, we consider the two cases (i) $\bfv$ is a convex vertex and
(ii) $\bfv$ is a concave vertex.
In case (i), referring to Figure~\ref{fig:convex} we see that
for $\bfx$ close enough to $\bfv$,
$\tau_1(\bfx) = \tau_2(\bfx) = 1$
and so
\begin{equation}\label{eq:limconvex}
 \lim_{\bfx \to \bfv} \frac{\gamma(\bfx)}{\phi(\bfx)} = 
   \epsilon.
\end{equation}
In case (ii), the values of $\tau_1(\bfx)$ and $\tau_2(\bfx)$
depend on the location of $\bfx$, even when $\bfx$ is close to $\bfv$.
However, for any $\bfx$ that is close enough to $\bfv$,
we have the identity (observed in \cite{Hormann:06})
$$
 \tau_1(\bfx)\alpha_1(\bfx) 
  + \tau_2(\bfx)\alpha_2(\bfx) 
    = \alpha_{[\bfy_1,\bfy_2]}(\bfx).
$$
This can be verified in the three cases
illustrated in Figure~\ref{fig:concave}.
In the three configurations, from left to right, we have, respectively,
$$
   \alpha_{[\bfy_1,\bfy_2]}(\bfx) =
   \begin{cases} \alpha_1(\bfx) + \alpha_2(\bfx), \cr 
                 \alpha_1(\bfx) - \alpha_2(\bfx), \cr 
               - \alpha_1(\bfx) + \alpha_2(\bfx).
   \end{cases}
$$
Thus,
$$
 \lim_{\bfx \to \bfv}
  (\tau_1(\bfx)\alpha_1(\bfx) 
  + \tau_2(\bfx)\alpha_2(\bfx))
    = \alpha_{[\bfy_1,\bfy_2]}(\bfv)
    = \alpha_{[\bfv_1,\bfv_2]}(\bfv) \in (0,\pi).
$$
Since
$\sin((\alpha_1(\bfx) + \alpha_2(\bfx))/2) \le 1$,
it follows that in case (ii),
\begin{equation}\label{eq:limconcave}
 \limsup_{\bfx \to \bfv} \frac{\gamma(\bfx)}{\phi(\bfx)}
 \le \frac{\epsilon}{\sin(\alpha_{[\bfv_1,\bfv_2]}(\bfv)/2)}.
\end{equation}
From (\ref{eq:limconvex}) and (\ref{eq:limconcave}) we deduce that
for any type of vertex $\bfv$,
$|g(\bfx) - f(\bfv)| \to 0$ as $\bfx \to \bfv$.
\end{proof}

\section{Numerical examples}

In this section we present two examples of transfinite mean value interpolants
of different functions $f(x,y)$
over a polygonal-shaped domain in order to confirm
the theoretical interpolation property proven in
Sections~\ref{sec:edge} and~\ref{sec:vertex}.
For the implementation we have
evaluated the mean value interpolant $g(x,y)$ using the
boundary integral formula of \cite{Dyken:09}.
This is more efficient than applying
the definition, equation \eqref{eq:mv}, which would
require computing intersection points.

The first function we consider is
$$ f(x,y) = x^2-y^2 $$ 
defined on the non-convex polygon in Figure \ref{ex1}a.
Figures \ref{ex1}a and \ref{ex1}b illustrate the exact surface and
Figures \ref{ex1}c and \ref{ex1}d
the corresponding interpolant $g(x,y)$.
Figure \ref{ex1}e shows the absolute error $|f(x,y)-g(x,y)|$.
The darker the colour the smaller the error and, as expected,
the error vanishes as we get close to the boundary.
\begin{figure}[ht]
\centering
\subfloat[]{
\includegraphics[width=5cm]{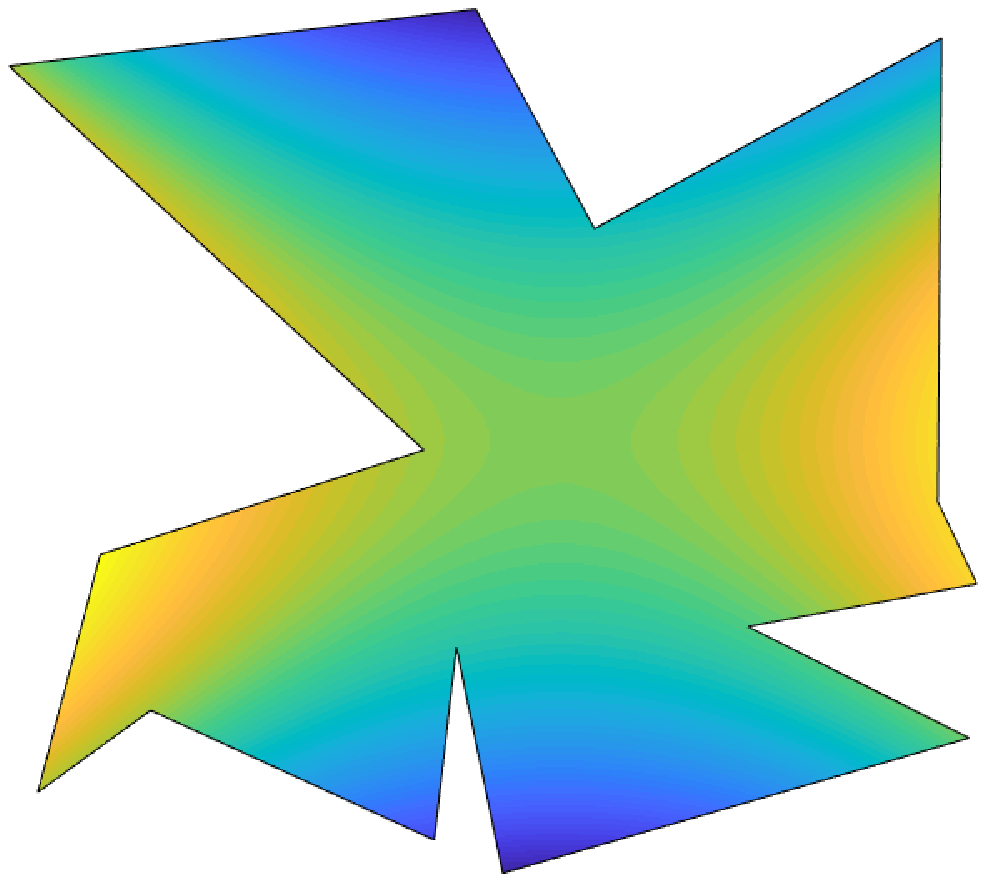}}\quad
\subfloat[]{\includegraphics[width=5cm]{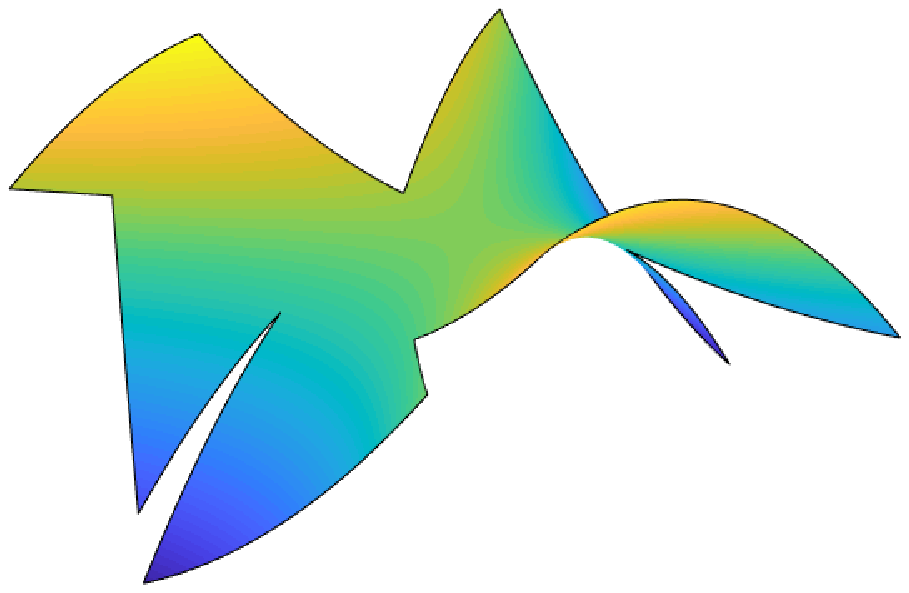}}\\
\subfloat[]{
\includegraphics[width=5cm]{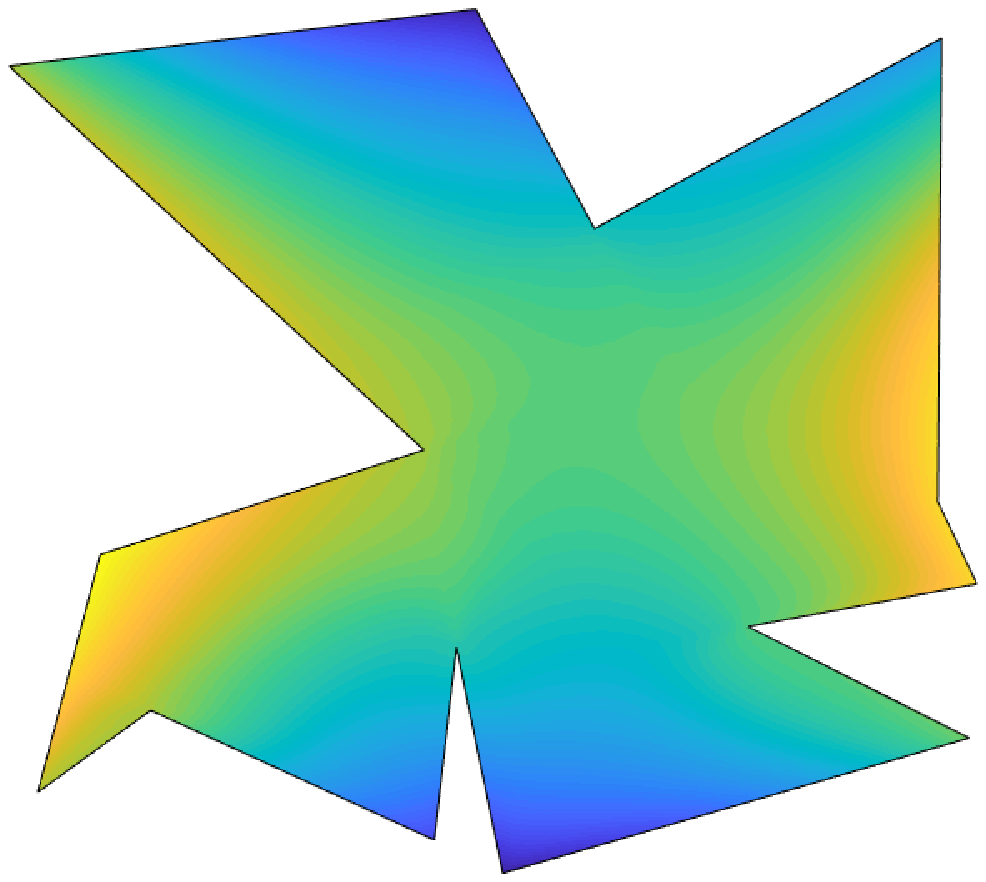}}\quad
\subfloat[]{\includegraphics[width=5cm]{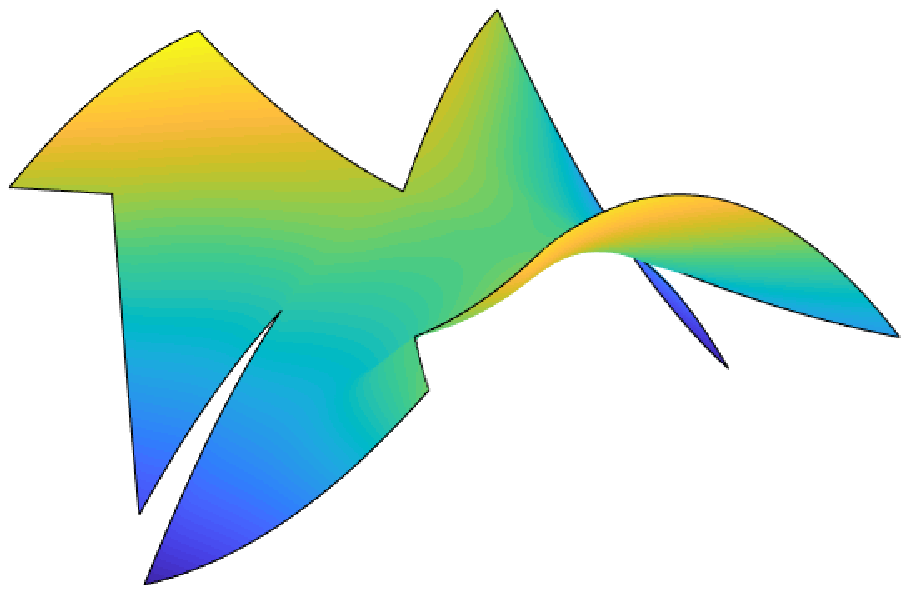}}\\
\subfloat[]{
\includegraphics[width=5cm]{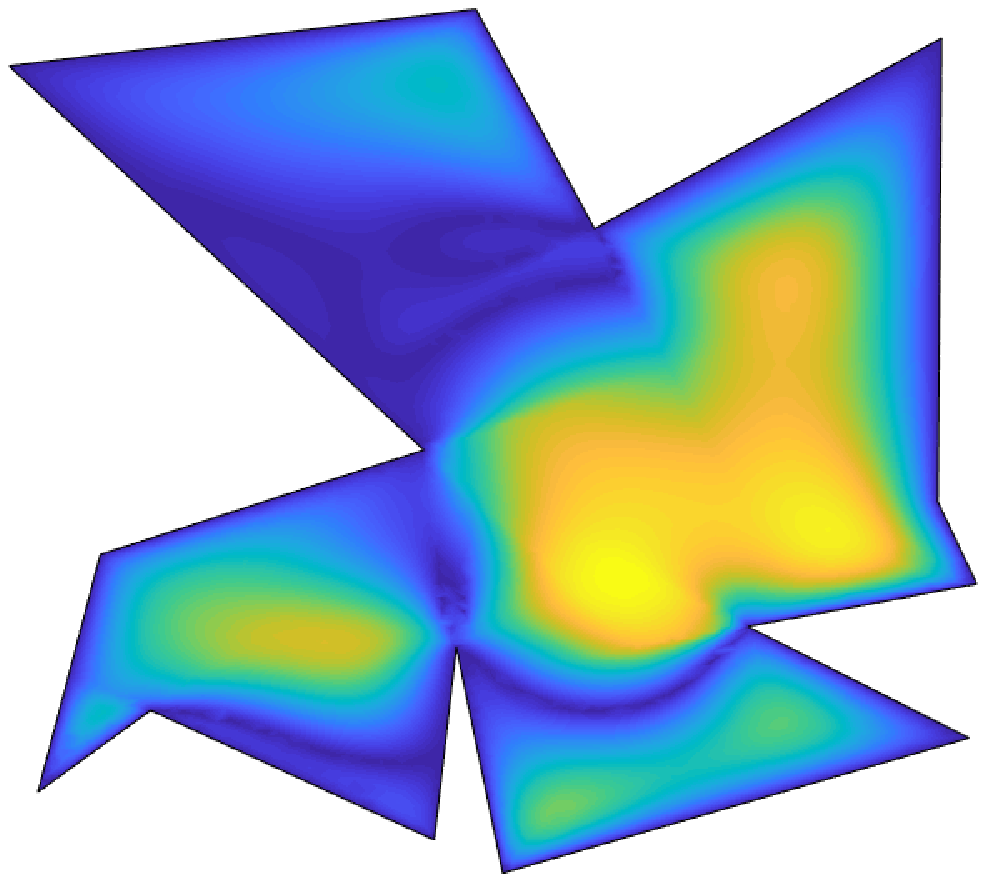}}
\caption{}\label{ex1}
\end{figure}\\
For our second example we chose the function
$$ f(x,y) = \frac{1}{9}[\tanh(9x-9y) +1]. $$
Figures \ref{ex2}a and \ref{ex2}b and
Figures \ref{ex2}c and \ref{ex2}d show the exact surface and the
interpolant, respectively, while Figure \ref{ex2}e shows the 
absolute error. 
\begin{figure}[ht]
\centering
\subfloat[]{
\includegraphics[width=5cm]{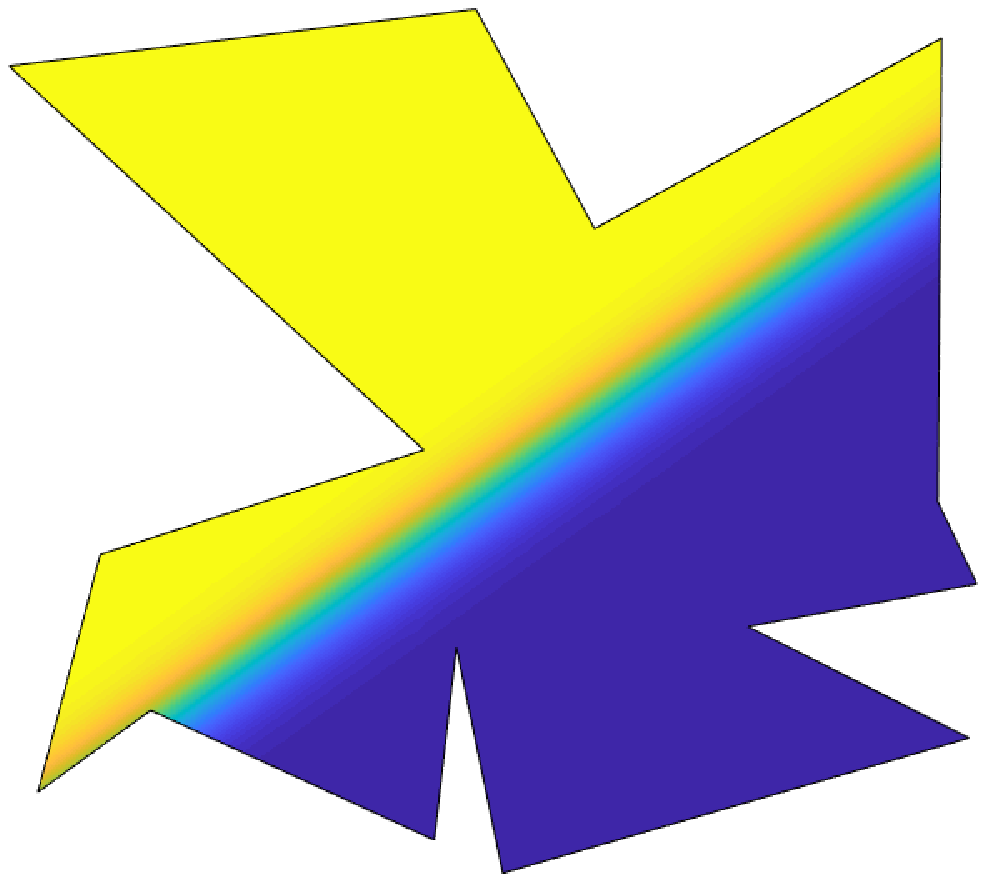}}\quad
\subfloat[]{\includegraphics[width=5cm]{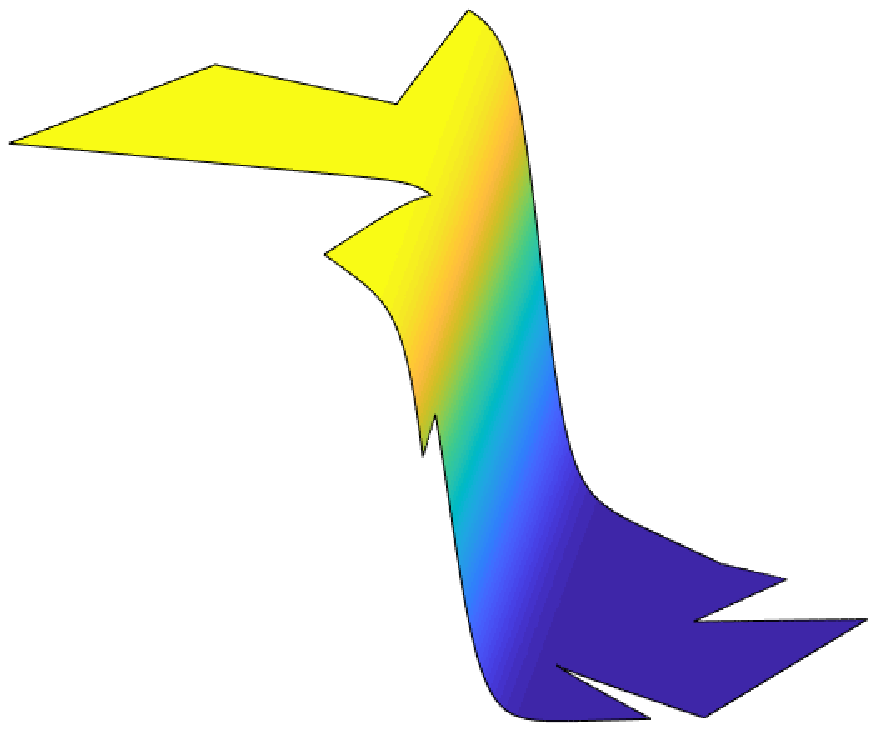}}\\
\subfloat[]{
\includegraphics[width=5cm]{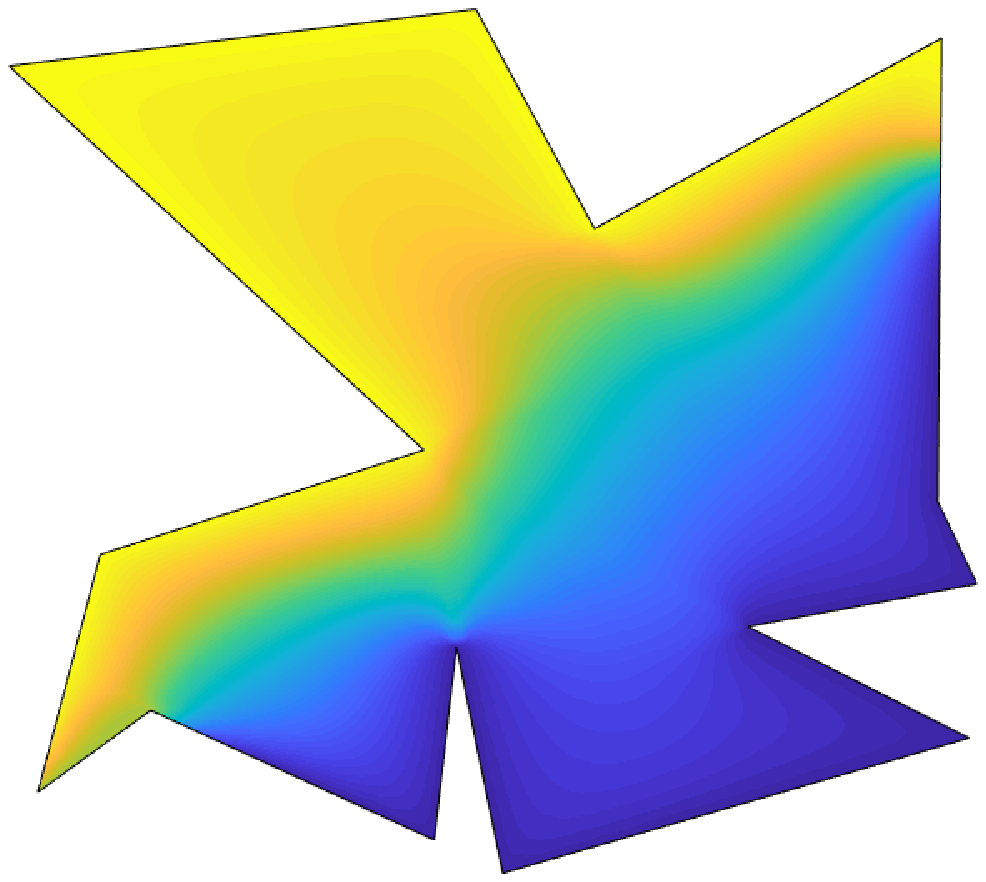}}\quad
\subfloat[]{\includegraphics[width=5cm]{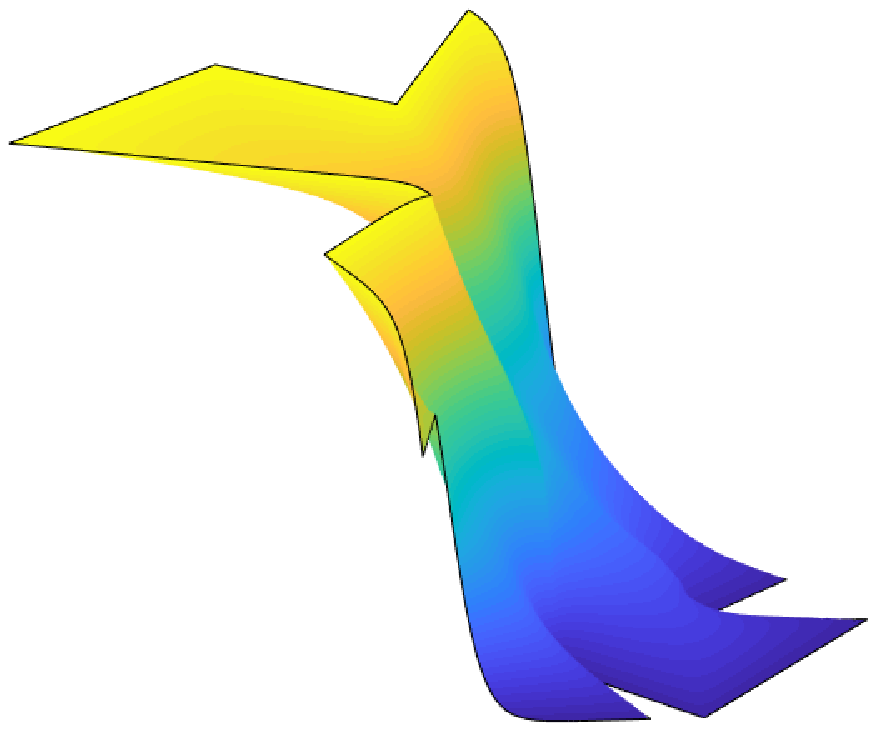}}\\
\subfloat[]{
\includegraphics[width=5cm]{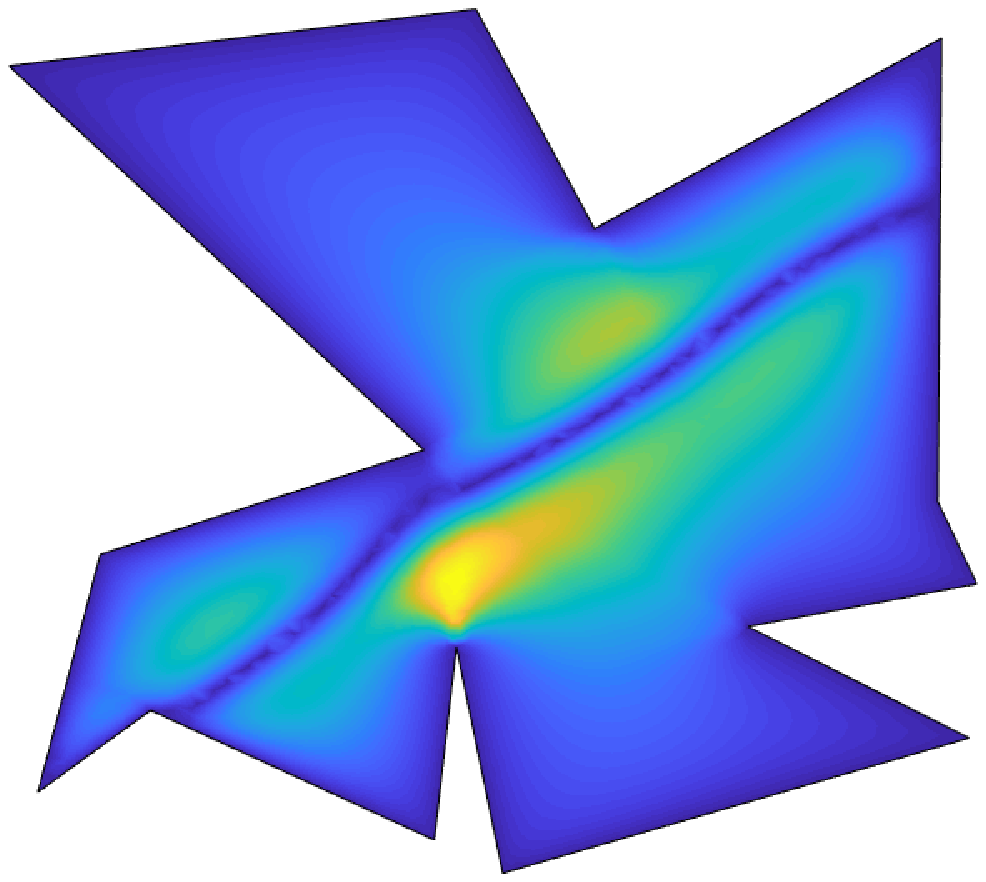}}
\caption{}\label{ex2}
\end{figure}

\Acknowledgement{
This project has received funding from the European Union's
Horizon 2020 research and innovation programme under the
Marie Sk{\l}odowska-Curie grant agreement No~675789.}

\bibliography{coordinates}

\end{document}